\documentclass{article}
\usepackage{amsfonts}
\usepackage{amsmath,amssymb,amsthm}
\usepackage{bbm}
\usepackage{mathrsfs}

\newtheorem{theorem}{Theorem}[section]
\newtheorem{lemma}[theorem]{Lemma}

\newtheorem{proposition}[theorem]{Proposition}
\newtheorem{corollary}[theorem]{Corollary}

\newtheorem{remark}[theorem]{Remark}
\newcommand{\To}{\longrightarrow}
\newcommand{\N}{\mathbb{N}}
\newcommand{\Z}{\mathbb{Z}}
\newcommand{\R}{\mathbb{R}}
\newcommand{\B}{\mathscr{B}}
\newcommand{\ww}[1]{\widetilde{#1}}

\begin{document}
	\title{{\bf Equilibrium states which are not Gibbs measure on hereditary subshifts} \thanks{$\mathsection$ School of Mathematical Sciences and Institute of Mathematics, Nanjing Normal University, Nanjing 210093, P.R.China. Email: zjlin137@126.com; School of Mathematical Sciences and Institute of Mathematics, Nanjing Normal University, Nanjing 210093, P.R.China, and Center of Nonlinear Science, Nanjing University, Nanjing 210093, P.R.China. Email: ecchen@njnu.edu.cn.}}
    \author{Zijie Lin and Ercai Chen$^\mathsection$ }
    \date{}
	\maketitle

\begin{abstract}
In this paper, we consider which kind of invariant measure on hereditary subshifts is not Gibbs measure. For the hereditary closure of a subshift $(X,S)$, we prove that in some situation, the invariant measure $\nu*B_{p,1-p}$ can not be a Gibbs measure where $\nu$ is an invariant measure on $(X,S)$. As an application, we show that for some $\B$-free subshifts, the unique equilibrium state $\nu_\eta*B_{p,1-p}$ is not Gibbs measure.
\end{abstract}

\section{Introduction}
Recall that a subshift $(X, S)$ is a subsystem of full shift $(\{0,1\}^\Z, \sigma)$ where $\{0,1\}^\Z=\{(x_i)_{i\in\Z}: x_i\in\{0,1\}\}$ and $\sigma: \{0,1\}^\Z\rightarrow \{0,1\}^\Z$ with $\sigma((x_i)_{i\in\Z})=(x_{i+1})_{i\in\Z}$. It means that $X$ is a closed $\sigma$-invariant subset of $\{0,1\}^\Z$ and $S=\sigma |_X$. Denote by $\mathcal{M}(X,S)$ (resp. $\mathcal{M}^e(X, S)$) the set of all the Borel $S$-invariant (resp. ergodic $S$-invariant) probability measure on $(X, S)$.

For a subshift $(X, S)$, recall that the set of all the $n$-length \emph{word} is the set $\mathcal{L}_n(X)=\{W=[w_0w_1\cdots w_{n-1}]: \text{ there exists }x\in X,\quad x_i=w_i\text{ for }i=0,1,...,n-1\}$ and the \emph{language} is the set $\mathcal{L}(X)=\bigcup_{n\in\N}\mathcal{L}_n(X)$. For each word $W\in\mathcal{L}(X)$, denote by $|W|$ the length of the word $W$, that is, $|W|=n$ if and only if $W\in\mathcal{L}_n(X)$. For a word $W\in\mathcal{L}(X)$ or a point $x\in X$, let $W[i,j]=[w_i\cdots w_j]$ and $x[i,j]=[x_i\cdots x_j]$ for any suitable $i\le j$. Each word $W$ also stands for the corresponding cylinder set $W=\{x\in X: x[0, |W|-1]=W\}$ with the same denotation. For any word $W$, define $\#_1W=\#\{1\le i<|W|: w_i=1\}$.

For two words $W=[w_0\cdots w_{n-1}],W'=[w'_0\cdots w'_{n-1}]\in \mathcal{L}_n(X)$, we call $W\le W'$ if $w_i\le w'_i$ for each $i=0,1,...,n-1$. Also, for two points $x=(x_i)_{i\in\Z},y=(y_i)_{i\in\Z}\in X$, we call $y\le x$ if $y_i\le x_i$ for each $i\in\Z$. The subshift $(X, S)$ is \emph{hereditary} if for any $W\in\mathcal{L}(X)$ and any $W'\le W$, the word $W'\in\mathcal{L}(X)$. Define the \emph{hereditary closure} of $(X,S)$ by
$$\widetilde{X}=\{y\in \{0,1\}^\Z: \text{there exists } x\in X \text{ such that } y\le x\}.$$
It follows that $(X,S)$ is hereditary if and only if $\widetilde{X}=X$. Examples of hereditary subshift include many $\B$-free systems, introduced in Section \ref{Bsys}. The basic properties of hereditary shifts are showed in \cite{Ker07, Kwi13}. In \cite{Prz15}, J.K.-Przymus, M. Lema\'nczyk and B. Weiss studied the invariant measure on $\B$-free subshifts. In \cite{Dym18}, A. Dymek, S. Kasjan, J.K.-Przymus and M. Lema\'nczyk studied entropy and intrinsic ergodicity of $\B$-free subshifts.

Equilibrium states play an important role on complicated physical systems. Bowen\cite{Bow75} and Ruelle\cite{Rue78} have studied the existence of equilibrium states for continuous functions on shifts of finite type. In \cite{Gur98}, the authors show that an equilibrium state exists if and only if the function is positive recurrent, and in this case the equilibrium state is unique. In \cite{Sar99}, the author shows the existence of equilibrium states for H\"older continuous positive recurrent functions for which the Ruelle--Perron--Frobenius operator maps the constant function $1$ to a bounded function.

Gibbs measures have strong relationship with the equilibrium states. The idea of Gibbs measures comes from statistical physics(\cite{Lan73,Rue99}). The basic properties of Gibbs measures were introduced in \cite{Bow08,Sin72}. In \cite{Bow75,Rue78}, the authors proved the existence of Gibbs measures on topological Markov shifts. In \cite{Mau01}, Mauldin and Urba\'nski found sufficient topological conditions for the existence of Gibbs measures. In \cite{Sar03}, the author showed that Mauldin and Urba\'nski's sufficiency result can be derived from the generalized Ruelle's Perron--Frobenius theorem of \cite{Sar01}, and gave a new proof of their result.

In \cite{JM20}, J.K.-Przymus and M. Lema\'nczyk proved that for some of hereditary subshifts, the maximal entropy measure does not have the Gibbs property(See details in \cite{JM20}).
This work motivates us to consider that for hereditary subshifts, when the equilibrium state is not Gibbs measure.

\begin{theorem}\label{t:1.1}
  For the hereditary closure $(\ww{X},S)$ of a subshift $(X,S)$, non-atomic measure $\nu\in\mathcal{M}^e(X,S)$ with $D_\nu=D$, $\kappa=\nu*B_{q,1-q}\in\mathcal{M}^e(\ww{X},S)$ with some $0<q<1$, and a continuous function $\ww{\phi}: \ww{X}\rightarrow \R$ with $D^{\ww{\phi}}_\kappa=D^{\ww{\phi}}$. If $$\ww{P}\le(\mathrm{Var}\ww{\phi}([0])-\log (1-q)-\mathrm{Var}\ww{\phi}([1]))d+d^{\ww{\phi}}-\mathrm{Var}\ww{\phi}([0]),$$
  $$\sup\ww{\phi}([1])\ge\sup\ww{\phi}([0]),$$
  and
  $$\mathrm{Var}\ww{\phi}([1])\le\mathrm{Var}\ww{\phi}([0])-\log (1-q),$$
   then $\kappa$ is not the Gibbs measure for $\ww{\phi}$.
\end{theorem}

As an application, we consider some $\B$-free systems which are shown that its unique equilibrium state is not Gibbs measure.
As a generalization of square-free numbers, $\B$-free numbers and $\B$-free systems were studied for several years (See details for \cite{Abd13,Dym18,Prz15}). Fix an infinite set $\B=\{b_1,b_2,\cdots \}\subset \{2,3,\cdots \}$. The set $\B$ is said to be \emph{pairwise coprime} if $\gcd(b_i, b_j)=1$ for any $i\neq j$. We consider $\B$ satisfies the following conditions:
\begin{equation}\label{condition1}
\B\text{ is infinite and pairwise coprime, and satisfies }\sum_{b\in\B}\frac1b<\infty.
\end{equation}
For example, $\B=\{p^2: p \text{ is prime number}\}$ satisfies the above condition. When $\B$ satisfies condition (\ref{condition1}), ergodic and topological properties of the corresponding $\B$-free systems were studied in \cite{Abd13,Dym18,Prz15}.

In the present paper, we prove the following theorem.

\begin{theorem}\label{t:1.2}
	Suppose that $\B=\{b_1,b_2,\cdots\}$ satisfies (\ref{condition1}) and $b_1=2$. For $\phi=a_{00}\mathbbm{1}_{[00]}+a_{01}\mathbbm{1}_{[01]}+a_{1}\mathbbm{1}_{[1]}$, the unique equilibrium state $\nu_\eta*B_{p,1-p}$ for $\phi$ is not Gibbs measure, where
	$$p=\frac{2^{2a_{00}}}{2^{a_1+a_{01}}+2^{2a_{00}}}.$$
\end{theorem}

This paper is organized as follows. In Section 2, we recall some basic notions and their properties. In Section 3, we introduce the densities for a continuous map $\phi$ and prove an inequality for them. Section 4 is the proof of Theorem \ref{t:1.1}. In Section 5, we prove Theorem \ref{t:1.2}, which gives some $\B$-free subshifts whose unique equilibrium state is not Gibbs measure, as an application of Theorem \ref{t:1.1}.

\section{Preliminaries}

For hereditary subshifts, we consider the invariant measure given by the following ways.
Let $Q:X\times \{0,1\}^\Z\To\widetilde{X}$ be the coordinatewise multiplication:
$$Q(x,y)=(...,x_{-1}y_{-1}, x_0y_0, x_1y_1,...)$$
for $x=(x_i)_{i\in\Z}\in X$ and $y=(y_i)_{i\in\Z}\in \{0,1\}^\Z$. For any $\nu\in\mathcal{M}(X,S)$ and $\mu\in\mathcal{M}(\{0,1\}^\Z,S)$ the \emph{multiplicative convolution} of $\nu$ and $\mu$ is the measure $\nu*\mu\in\mathcal{M}(\widetilde{X}, S)$ given by:
$$\nu*\mu=(\nu\otimes\mu)\circ Q^{-1}.$$

For a subshift $(X,S)$, the \emph{topological entropy} $h=h(X,S)$ is defined as follows:
$$h(X,S)=\lim_{n\rightarrow\infty}\frac{\log\#\mathcal{L}_n(X)}{n}.$$
And for each $\mu\in\mathcal{M}(X,S)$, the \emph{measure entropy} of $\mu$ is defined as follows:
$$h_\mu(X,S)=\lim_{n\rightarrow\infty}\frac{h_\mu(\mathcal{L}_n(X))}{n},$$
where $h_\mu(\mathcal{L}_n(X))=-\sum_{W\in\mathcal{L}_n(X)}\mu(W)\log\mu(W)$. By the variational principle, $h(X,S)=\sup_{\mu\in\mathcal{M}(X,S)}h_\mu(X,S)$.

For a subshift $(X,S)$ and a continuous function $\phi: X\rightarrow \R$, the \emph{topological pressure} $P=P(X,\phi)$ is defined as follows:
$$P(X,\phi)=\lim_{n\rightarrow\infty}\frac{\log Z_n(X,S,\phi)}{n}$$
where $Z_n(X,S,\phi)=\sum_{W\in\mathcal{L}_n(X)}2^{\sup_{x\in W}\sum_{i=0}^{n-1}\phi(S^ix)}$. By the variational principle, $$P(X,\phi)=\sup_{\mu\in\mathcal{M}(X,S)}\left(h_\mu(X,S)+\int \phi d\mu\right).$$
The measure $\mu$ is called \emph{equilibrium state} if it satisfies $P(X,\phi)=h_\mu(X,S)+\int \phi d\mu$,

In \cite{JM20}, a measure $\mu\in\mathcal{M}^e(X,S)$ is said to have \emph{Gibbs Property}, if there exists $a>0$ such that for any $\mu$-positive measure block $C$,
$$\mu(C)\ge a\cdot2^{-|C|h(X,S)}.$$

In \cite{Bow08}, for a continuous function $\phi: X\rightarrow \R$, a measure $\mu\in\mathcal{M}(X,S)$ is called a \emph{Gibbs measure} for $\phi$, if there exist $P=P(X,\phi)\ge0$ and $c=c(X,\phi)>0$ such that for any $n\in \N$, $\mu$-positive measure block $C$ of length $n$ and $x\in C$,
$$c^{-1}\le \frac{\mu(x[0,n-1])}{2^{\sum_{i=0}^{n-1}\phi(S^ix)-nP(X,\phi)}}\le c.$$
The constant $P=P(X,\phi)$ above is the topological pressure of $\phi$ on $(X, S)$.

\section{Densities for a continuous function}

In \cite{JM20}, it defines four notions of density. For a subshift $(X, S)$, let
$$d=\sup_{\mu\in\mathcal{M}(X,S)}\mu([1]),$$
$$D=\lim_{n\rightarrow\infty}\frac1n\max_{W\in\mathcal{L}_n(X)}\#_1W.$$
For $\mu\in\mathcal{M}(X,S)$, let
$$d_\mu=\mu([1]),$$
$$D_\mu=\lim_{n\rightarrow\infty}\frac1n\max_{W\in\mathcal{L}_n(X),\mu(W)>0}\#_1W.$$

Similar with the above definitions, we define four notions of density for a continuous map. For a continuous map $\phi: X\rightarrow \R$, let
$$
D^\phi=\lim_{n\rightarrow\infty}\frac1n\sup_{x\in X}\sum_{i=0}^{n-1}\phi(S^ix),
$$
$$d^\phi=\sup_{\mu\in\mathcal{M}^e(X,S)}\int_X\phi d\mu.$$

For $\mu\in\mathcal{M}(X,S)$, define
$$
D_\mu^\phi=\lim_{n\rightarrow\infty}\frac1n\max_{W\in\mathcal{L}_n(X),\mu(W)>0}\sup_{x\in W}\sum_{i=0}^{n-1}\phi(S^ix),
$$
$$d_\mu^\phi=\int_X\phi d\mu.$$
Both $D^\phi$ and $D_\mu^\phi$ are exist, because for any $n,m\in\N$,
\begin{equation*}
\begin{aligned}
&\sup_{x\in X}\sum_{i=0}^{n+m-1}\phi(S^ix)\\
=&\sup_{x\in X}\left( \sum_{i=0}^{n-1}\phi(S^ix)+\sum_{i=0}^{m-1}\phi(S^{n+i}x)\right) \\
\le& \sup_{x\in X}\sum_{i=0}^{n-1}\phi(S^ix)+\sup_{x\in X}\sum_{i=0}^{m-1}\phi(S^ix)
\end{aligned}
\end{equation*}
and
\begin{equation*}
\begin{aligned}
&\max_{W\in\mathcal{L}_{n+m}(X),\mu(W)>0}\sup_{x\in W}\sum_{i=0}^{n+m-1}\phi(S^ix)\\
=&\max_{W\in\mathcal{L}_{n+m}(X),\mu(W)>0}\sup_{x\in W}\left( \sum_{i=0}^{n-1}\phi(S^ix)+\sum_{i=0}^{m-1}\phi(S^{n+i}x)\right) \\
\le& \max_{W\in\mathcal{L}_{n+m}(X),\mu(W)>0}\left( \sup_{x\in W[0,n-1]}\sum_{i=0}^{n-1}\phi(S^ix)+\sup_{x\in W[n,n+m-1]}\sum_{i=0}^{m-1}\phi(S^{n+i}x)\right)\\
\le& \max_{W\in\mathcal{L}_n(X),\mu(W)>0}\sup_{x\in W}\sum_{i=0}^{n-1}\phi(S^ix)+\max_{W\in\mathcal{L}_m(X),\mu(W)>0}\sup_{x\in W}\sum_{i=0}^{m-1}\phi(S^ix).
\end{aligned}
\end{equation*}
The last inequality is hold because if $W\in\mathcal{L}_{n+m}(X)$ and $\mu(W)>0$,  then $\mu(W[0,n-1])\ge\mu(W)>0$ and $\mu(S^nW[n,n+m-1])=\mu(W[n,n+m-1])\ge\mu(W)>0$. Therefore, by subadditivity, $D^\phi$ and $D^\phi_\mu$ are exist.

It is obvious that when $\phi=\mathbbm{1}_{[1]}$, we have $d=d^\phi$, $D=D^\phi$, $d_\mu=d_\mu^\phi$ and $D_\mu=D_\mu^\phi$ for any $\mu\in\mathcal{M}(X,S)$.

It is proved in \cite{JM20} that $d_\mu\le D_\mu \le D =d$. Similarly, we also prove the corresponding theorem for the four notions of density for $\phi$.

\begin{theorem}
	For any $\mu\in\mathcal{M}^e(X,S)$ and any continuous function $\phi: X\rightarrow \R$, we have $d_\mu^\phi\le D_\mu^\phi \le D^\phi=d^\phi$.
\end{theorem}
\begin{proof}
	\item{(1)} $D^\phi=d^\phi$: For any $n\in\N$, let $x^{(n)}$ satisfy
	$$\sum_{i=0}^{n-1}\phi(S^ix^{(n)})=\sup_{x\in X}\sum_{i=0}^{n-1}\phi(S^ix).$$
	Let
	$$\mu_n=\frac1n\sum_{i=0}^{n-1}\delta_{S^ix^{(n)}}.$$
	Without loss of generality, we can assume $\mu_n\rightarrow \mu$, so
	$$\frac1n\sum_{i=0}^{n-1}\phi(S^ix^{(n)})=\int_X\phi d\mu_n\rightarrow \int_X\phi d\mu\le d^\phi.$$
	Therefore, $D^\phi\le d^\phi$.
	Let $\nu$ satisfy $\int_X\phi d\nu=\sup_{\mu\in\mathcal{M}^e(X,S)}\int_X\phi d\mu$ and $x$ is a generic point of $\nu$, then
	$$d^\phi=\int_X\phi d\nu=\lim_{n\rightarrow\infty}\frac1n\sum_{i=0}^{n-1}\phi(S^ix)\le D^\phi.$$
	
	\item{(2)} $D^\phi_\mu\le D^\phi$: By the definition, it is obvious.
	
	\item{(3)} $d^\phi_\mu\le D^\phi_\mu$:
	For any $\epsilon>0$, fix large enough $n$ such that
	$$\max_{W\in\mathcal{L}_n(X),\mu(W)>0}\sup_{x\in W}\sum_{i=0}^{n-1}\phi(S^ix)<n(D^\phi+\epsilon).$$
	Fix $x$ is a generic point of $\mu$. For $i\in\N$, define\\
	(a)$i$ is good: $\mu(x[i,i+n-1])>0;$\\
	(b)$i$ is bad: $\mu(x[i,i+n-1])=0.$
	
	Set $i_0=-n$. For $j=1,2,...$, define inductively that
	$$i_j=\min\{i\ge i_{j-1}+n: i \text{ is good}\}.$$
	So for any $i\in \cup_{j=1}[i_{j-1}+n,i_j-1]$, $i$ is bad.
	For any $k\in \N$, because $\{[i_j,i_j+n-1]: j=1,2,...\}$ is pairwise disjoint, we have
	$$\#\{i_j\in [0,k-1]: j=1,2,...\}\le\frac kn+1.$$
	In addition,
	$$\frac1k\#\{i\in [0,k-1]: i \text{ is bad}\}\le\frac1k\sum_{i=0}^{k-1}\sum_{W\in\mathcal{L}_n(X),\mu(W)=0}1_W(S^ix)\rightarrow 0.$$
	So let $K\in\N$ such that if $k>K$, then
	$$\#\{i\in [0,k-1]: i \text{ is bad}\}\le \epsilon k.$$
	Therefore,
	\begin{equation*}
	\begin{aligned}
	&\frac1k\sum_{i=0}^{k-1}\phi(S^ix)\\
	\le& \frac1k\left(\sum_{j\in\{j: i_j\in[0,k-1]\}}\sum_{i=0}^{n-1}\phi(S^{i_j+i}x)+\sum_{i\in[0,k-1], i\text{ is bad}}\phi(S^ix)\right) \\
	\le& \frac1k((\frac kn+1)(nD^\phi_\mu+\epsilon)+\epsilon k|\phi|)\\
	\le& D^\phi_\mu+\epsilon(1+|\phi|)+\frac{nD^\phi_\mu+\epsilon}{k},
	\end{aligned}
	\end{equation*}
	where $|\phi|=\sup_{x\in X}|\phi(x)|$. Let $k\rightarrow\infty$, we have $d^\phi_\mu\le D^\phi_\mu+\epsilon(1+|\phi|)$. By the arbitrariness of $\epsilon$, it shows that $d^\phi_\mu\le D^\phi_\mu$.
	
	By all of above, it ends the proof.
\end{proof}

\section{Proof of Theorem \ref{t:1.1}}

For convenience, we prove the case of $q=1/2$. Let $\kappa=\nu*B_{1/2,1/2}$, where $\nu\in\mathcal{M}^e(X,S)$ and $B_{q,1-q}$ stands for the Bernoulli measure on $\{0,1\}^\Z$ with $B_{q,1-q}([0])=q$ and $B_{q,1-q}([1])=1-q$.

For the hereditary closure $(\ww{X},S)$ of a subshift $(X,S)$ and a continuous map $\ww{\phi}: \ww{X}\rightarrow \R$, denote by $\ww{P}=P(\ww{X},\ww{\phi})$ the topological pressure for $\ww{\phi}$ on $(\ww{X}, S)$.

Here, we need some lemmas in \cite{JM20}.

\begin{lemma}[\cite{JM20}]\label{JMlem1}
  Let $\nu\in\mathcal{M}(X,S)$. Then for $\kappa=\nu * B_{1/2,1/2}$, we have
  $$\kappa(C)=\sum_{C\le C'\in\mathcal{L}(X)}\nu(C')\cdot2^{-\#_1C'}$$
  for each $C\in\mathcal{L}(\ww{X})$.
\end{lemma}

\begin{lemma}[\cite{JM20}]\label{JMlem2}
  Let $\nu\in\mathcal{M}(X,S)$ and $a>0$. Suppose that there is a sequence of block $C_n$ such that $|C_n|\nearrow \infty$ and $\nu(C_n)\ge a$. Then there exists $(n_k)$ such that $\bigcap_{k\ge 1}C_{n_k}\neq\emptyset$. Moreover, we have $\nu(\{x\})\ge a$ for $\{x\}=\bigcap_{k\ge 1}C_{n_k}$.
\end{lemma}

For a continuous map $\ww{\phi}: \ww{X}\rightarrow \R$ and $A\subset \ww{X}$, let $\mathrm{Var}\ww{\phi}(A)=\sup\ww{\phi}(A)-\inf\ww{\phi}(A)$.

\begin{theorem}\label{mainth}
  For the hereditary closure $(\ww{X},S)$ of a subshift $(X,S)$, non-atomic measure $\nu\in\mathcal{M}^e(X,S)$ with $D_\nu=D$, $\kappa=\nu*B_{1/2,1/2}\in\mathcal{M}^e(\ww{X},S)$ and a continuous map $\ww{\phi}: \ww{X}\rightarrow \R$ with $D^{\ww{\phi}}_\kappa=D^{\ww{\phi}}$. If $$\ww{P}\le(1+\mathrm{Var}\ww{\phi}([0])-\mathrm{Var}\ww{\phi}([1]))d+d^{\ww{\phi}}-\mathrm{Var}\ww{\phi}([0]),$$
  $$\sup\ww{\phi}([1])\ge\sup\ww{\phi}([0]),$$
  and
  $$\mathrm{Var}\ww{\phi}([1])\le \mathrm{Var}\ww{\phi}([0])+1,$$
   then $\kappa$ is not the Gibbs measure for $\ww{\phi}$.
\end{theorem}
\begin{proof}
  If $\kappa=\nu*B_{1/2,1/2}$ is Gibbs measure for $\widetilde{\phi}$. Because $D^{\ww{\phi}}_\kappa=D^{\ww{\phi}}$ , for any $n\in\N$, there exists $x^{(n)}\in \ww{X}$ such that
$$\sum_{i=0}^{n-1}\ww{\phi}(S^ix^{(n)})
=\max_{W\in\mathcal{L}_n(\ww{X}),\kappa(W)>0}\sup_{y\in W}\sum_{i=0}^{n-1}\widetilde{\phi}(S^iy)
\ge nD^{\widetilde{\phi}}$$ and $\kappa(x^{(n)}[0,n-1])>0$.

Since $D_\nu=D$, for any $n\in\N$, there exists $C_n\in \mathcal{L}_n(X)$ such that
$\#_1C_n=\max_{W\in \mathcal{L}_n(X), \nu(W)>0}\#_1W\ge nD=nd$.

Let $a_n=\#_1C_n-\#_1x^{(n)}[0,n-1]$. By Lemma \ref{JMlem1},
$$\kappa(x^{(n)}[0,n-1])=\sum_{W\in \mathcal{L}_n(X),W\ge x^{(n)}[0,n-1]}\nu(W)\cdot2^{\#_1W}>0.$$
So there exists $W$ with $\nu(W)>0$ such that $\#_1W\ge \#_1x^{(n)}[0,n-1]$, which implies that $a_n\ge0$.
Now fix $y\in C_n$,
\begin{equation*}
\begin{aligned}
0\le& \sum_{i=0}^{n-1}\ww{\phi}(S^ix^{(n)})-\sum_{i=0}^{n-1}\ww{\phi}(S^iy)\\
\le& \#_1x^{(n)}[0,n-1]\sup\ww{\phi}([1])+(n-\#_1x^{(n)}[0,n-1])\sup\ww{\phi}([0])\\
&-\#_1C_n\inf\ww{\phi}([1])-(n-\#_1C_n)\inf\ww{\phi}([0])\\
=&(\#_1C_n-a_n)\sup\ww{\phi}([1])+(n-\#_1C_n+a_n)\sup\ww{\phi}([0])\\
&-\#_1C_n\inf\ww{\phi}([1])-(n-\#_1C_n)\inf\ww{\phi}([0])\\
=&\#_1C_n(\mathrm{Var}\ww{\phi}([1])-\mathrm{Var}\ww{\phi}([0]))\\
&+n\mathrm{Var}\ww{\phi}([0])-a_n(\sup\ww{\phi}([1])-\sup\ww{\phi}([0]))\\
\le&\#_1C_n(\mathrm{Var}\ww{\phi}([1])-\mathrm{Var}\ww{\phi}([0]))+n\mathrm{Var}\ww{\phi}([0]).
\end{aligned}
\end{equation*}

Therefore,
\begin{equation*}
\begin{aligned}
c^{-1}\le &\kappa(C_n)\cdot2^{n\ww{P}-\sum_{i=0}^{n-1}\ww{\phi}(S^iy)}\\
\le &\kappa(C_n)\cdot2^{n\ww{P}-\sum_{i=0}^{n-1}\ww{\phi}(S^ix^{(n)})}\\
&\cdot 2^{\#_1C_n(\mathrm{Var}\ww{\phi}([1])-\mathrm{Var}\ww{\phi}([0]))+n\mathrm{Var}\ww{\phi}([0])}\\
\le &\nu(C_n)\cdot2^{-\#_1C_n}\cdot2^{n\ww{P}-nd^{\ww{\phi}}}\\
&\cdot 2^{\#_1C_n(\mathrm{Var}\ww{\phi}([1])-\mathrm{Var}\ww{\phi}([0]))+n\mathrm{Var}\ww{\phi}([0])}\\
\le &\nu(C_n)\cdot2^{-nd(1-\mathrm{Var}\ww{\phi}([1])+\mathrm{Var}\ww{\phi}([0]))+n\ww{P}-nd^{\ww{\phi}}+n\mathrm{Var}\ww{\phi}([0])}\\
\le &\nu(C_n).
\end{aligned}
\end{equation*}

By Lemma \ref{JMlem2}, $\nu$ is atomic, which is a contradiction.
\end{proof}

If $\ww{\phi}=a_0\mathbbm{1}_{[0]}+a_1\mathbbm{1}_{[1]}$ for some $a_0\le a_1$, then $\mathrm{Var}\ww{\phi}([0])=\mathrm{Var}\ww{\phi}([1])=0$. So we have the following corollary.

\begin{corollary}\label{maincor}
  For the hereditary closure $(\ww{X},S)$ of a subshift $(X,S)$, non-atomic measure $\nu\in\mathcal{M}^e(X,S)$ with $D_\nu=D$ and $\kappa=\nu*B_{1/2,1/2}\in\mathcal{M}^e(\ww{X},S)$, suppose that $\ww{\phi}=a_0\mathbbm{1}_{[0]}+a_1\mathbbm{1}_{[1]}$ with $a_0\le a_1$ and $D^{\ww{\phi}}_\kappa=D^{\ww{\phi}}$. If $\ww{P}\le d+d^{\ww{\phi}}$, then $\kappa$ is not the Gibbs measure for $\ww{\phi}$.
\end{corollary}

More than Lemma \ref{JMlem1}, we prove that:

\begin{lemma}\label{JMlem1.1}
  Let $\nu\in\mathcal{M}(X,S)$ and $0<q<1$. Then for $\kappa=\nu * B_{q,1-q}$, we have
  $$\kappa(C)=\sum_{C\le C'\in\mathcal{L}(X)}\nu(C')\cdot q^{\#_1C'-\#_1C}(1-q)^{\#_1C}$$
  for each $C\in\mathcal{L}(\ww{X})$.
\end{lemma}

\begin{proof}
  For any $n\in\N$ and any $C\in \mathcal{L}_n(\ww{X})$, we have
  $$Q^{-1}(C)=\bigcup_{C\le C'\in\mathcal{L}(X)}\bigcup_{C'\cdot D=C,D\in\mathcal{L}_n(\{0,1\}^\Z)}C'\times D.$$
  For each $C'\ge C$, if $C'[i]=1$, $D[i]=C[i]$. So $\#\{i: D[i]=1 \text{ and } C'[i]=1\}=\#_1C$.
  On the other hand, if $C'[i]=0$, $D[i]$ is arbitrary. So $\#\{i: D[i]=1 \text{ and } C'[i]=0\}$ is ranged over $0$ to $n-\#_1C'$.
  Then
  \begin{equation*}
    \begin{aligned}
      \kappa(C)=&\sum_{C\le C'\in\mathcal{L}(X)}\sum_{C'\cdot D=C,D\in\mathcal{L}_n(\{0,1\}^\Z)}
      \nu(C')\cdot q^{n-\#_1D}(1-q)^{\#_1D}\\
      =&\sum_{C\le C'\in\mathcal{L}(X)}\nu(C')
      \sum_{i=0}^{n-\#_1C'}\binom{n-\#_1C'}{i}q^{n-\#_1C-i}(1-q)^{\#_1C+i}\\
      =&\sum_{C\le C'\in\mathcal{L}(X)}\nu(C')q^{\#_1C'-\#_1C}(1-q)^{\#_1C}.
    \end{aligned}
  \end{equation*}
\end{proof}

\begin{proof}[Proof of Theorem \ref{t:1.1}]
	By Lemma \ref{JMlem1.1}, if $C\in\mathcal{L}_n(X)$ attaches the maximum of the number of ones, that is, $\#_1 C= \max_{W\in\mathcal{L}_n(X)}\#_1W$, then
	$$\nu * B_{q,1-q}(C)=\nu(C)\cdot (1-q)^{\#_1C}.$$
	
	Therefore, Theorem \ref{t:1.1} can be proved by a similar proof of Theorem \ref{mainth}.
\end{proof}

\newcommand{\supp}[1]{\text{supp}(#1)}
\newcommand{\de}{\boldsymbol{\delta}}

\section{$\B$-free systems}\label{Bsys}

In this section, we consider some $\B$-free systems as an application of Theorem \ref{t:1.1} and Theorem \ref{mainth}. Firstly, we show some basic notions about $\B$-free systems.

Let $\B=\{b_1,b_2,\cdots \}$ be an infinite subset of $\{2,3,\cdots \}$.
In the rest of this section, we always assume that $\B$ satisfies condition (\ref{condition1}).

For $A\subset \Z$, define the densities of the positive part of $A$:
$$\text{lower density: }\underline{d}(A)=\liminf_{N\to\infty}\frac{\#A\cap[1,N]}{N},$$
$$\text{upper density: }\bar{d}(A)=\limsup_{N\to\infty}\frac{\#A\cap[1,N]}{N}.$$
If $\underline{d}(A)=\bar{d}(A)$, we set $d(A):=\underline{d}(A)=\bar{d}(A)$, called the \emph{density} of $A$.
Also, the lower logarithmic density $\underline{\de}(A)$ and the upper logarithmic density $\bar{\de}(A)$ of $A$ is defined as follows:
$$\underline{\de}(A)=\liminf_{N\to\infty}\frac{1}{\log N}\sum_{1\le a\le N,a\in A}\frac1a,$$
$$\bar{\de}(A)=\limsup_{N\to\infty}\frac{1}{\log N}\sum_{1\le a\le N,a\in A}\frac1a.$$
If $\underline{\de}(A)=\bar{\de}(A)$, we set $\de(A):=\underline{\de}(A)=\bar{\de}(A)$, called the \emph{logarithmic density} of $A$.

Let $\mathcal{M}_\B=\bigcup_{b\in\B}b\Z$, and $\mathcal{F}_\B=\Z\setminus\mathcal{M}_\B$. By our assumptions of $\B$ and \cite{Erd67,Hal83}, the density of $\mathcal{M}_\B$ exists, which means that $\B$ is \emph{Besicovitch}. In Section 2 of \cite{Dym18}, since $\sum_{b\in\B}\frac1b<\infty$(called \emph{thin} in \cite{Dym18}), $\B$ has light tails(See the definition in \cite{Dym18}), which implies that $\B$ is \emph{taut}, that is, $\de(\mathcal{M}_\B)>\de(\mathcal{M}_{\B\setminus\{b\}})$ for any $b\in\B$(\cite{Hal96}).

Let $\eta=\mathbbm{1}_{\mathcal{F}_\B}\in\{0,1\}^\Z$, that is,

$$\eta[n]=1 \text{ if and only if } n\in\mathcal{F}_\B.$$

Let
$$X_\eta=\{y\in\{0,1\}^\Z: \text{ for any }i,j\in\N,\quad y[i,i+j]=\eta[k,k+j]\text{ for some }k\}.$$
Recall that a point $y\in\{0,1\}^\Z$ is \emph{$\B$-admissible} if $\#(\supp y\text{ mod }b)<b$ for each $b\in\B$, where $\supp y=\{n\in\Z: y[n]=1\}$.

\begin{lemma}[\cite{Erd67, Sar12}]
  The space $X_\eta=X_\B:=\{y\in\{0,1\}^\Z: y\text{ is }\B\text{-admissible}\}$.
\end{lemma}

In particular, $X_\eta$ is hereditary, that is $\ww{X}_\eta=X_\eta$.

The following theorems and proposition are proved in \cite{Abd13,Dav36,Dav51,Prz15}.

\begin{theorem}[Theorem 5.3 in \cite{Abd13}]
  The topological entropy of the subshift $X_\B$ is given by
  $$h_{\mathrm{top}}(X_\B)=\prod_{i\in\N}\left(1-\frac{1}{b_i}\right).$$
\end{theorem}

\begin{theorem}[\cite{Dav36,Dav51}]\label{ddK}
  For any $\B\subset\N$, the logarithmic density $\de(\mathcal{M}_\B)$ of $\mathcal{M}_\B$ exist. Moreover,
  $$\de(\mathcal{M}_\B)=\underline{d}(\mathcal{M}_\B)=\lim_{K\rightarrow\infty}d(\mathcal{M}_{\{b\in\B:b\le K\}}).$$
\end{theorem}

\begin{proposition}[Proposition K in \cite{Dym18}]
  For any $\B\subset\N$, we have $h_{\mathrm{top}}(\ww{X}_\eta)=h_{\mathrm{top}}(X_\B)=\de(\mathcal{F}_\B)$.
\end{proposition}

By our assumption of $\B$, for subshift $(X_\B, S)$, we have
$$d=d(\mathcal{F}_\B)=h_{\mathrm{top}}(X_\B)=\prod_{i\in\N}\left(1-\frac{1}{b_i}\right).$$

Recall some known facts about the dynamical systems associated to $\B$-free numbers.
Let
$$\Omega:=\prod_{i\ge 1}\Z/b_i\Z=\{\omega=(\omega(1),\omega(2),\dots):\omega(i)\in\Z/b_i\Z,\, i=1,2,\dots\}.$$
With the product topology and the coordinatewise addition $\Omega$ becomes a compact metrizable Abelian group. Let $\mathbb{P}$ be the normalized Haar measure of $\Omega$ (which is the product of uniform measures on $\Z/b_k\Z$). Denote $T:\Omega\To\Omega$ the homeomorphism given by
$$T\omega=(\omega(1)+1,\omega(2)+1,\dots)$$
where $\omega=(\omega(1),\omega(2),\dots)$. It is known that $(\Omega, T)$ has zero entropy. Define $\varphi:\Omega\To\{0,1\}^\Z$ by
$$\varphi(\omega)(n)=\left\{
	\begin{aligned}
	1,&\text{ if for any }i\ge 1,\omega(i)+n\neq 0 \text{ mod } b_i,\\
	0,&\text{ otherwise.}
	\end{aligned}
	\right.
$$
It is not hard to see that $\varphi$ is Borel, equivariant(that is, $\varphi\circ T=S\circ\varphi$) and $\eta=\varphi(0,0,\dots)$. Let $\nu_\eta=\varphi_*(\mathbb{P}):=\mathbb{P}\circ\varphi^{-1}$ be the image of $\mathbb{P}$ via $\varphi$, which is called the \emph{Mirsky measure} of $(X_\eta,S)$. By \cite{Dym18}, $\eta$ is generic point of the Mirsky measure $\nu_\eta$. So for any $A\subset X_\eta$,
$$\nu_\eta(A)=\lim_{n\to\infty}\frac1n\sum_{i=0}^{n-1}\mathbbm{1}_A(S^i\eta).$$

Since $\B$ is infinite, by the tautness of $\B$ and Proposition 3.5 in \cite{Prz15}, $\nu_\eta$ is non-atomic. Also, in \cite{Kel19}, it is proved that the tautness of $\B$ implies that $\nu_\eta$ is full support on $X_\eta$.

Now, we turn to focus on some $\B$-free systems, and show that for some $\phi$, its unique equilibrium state is not Gibbs measure.

Theorem \ref{t:1.2} will be proved by several steps. First, to calculate the topological pressure on $(X_\eta,S)$, we need the following lemma, proved in \cite{Hal83}.

\begin{lemma}[\cite{Hal83}, p.242]\label{dK}
  For any $b_k$, $k\ge1$, any $r_k\in\Z/b_k\Z$, and $K\ge 1$, we have
  $$d\left(\bigcup_{k=1}^K(b_k\Z+r_k)\right)\ge d(\mathcal{M}_{\{b_1,\cdots, b_K\}}).$$
\end{lemma}

\begin{proposition}\label{PonX}
  Suppose that $\B=\{b_1,b_2,\cdots\}$ satisfies (\ref{condition1}) and $b_1=2$. For $\phi=a_{00}\mathbbm{1}_{[00]}+a_{01}\mathbbm{1}_{[01]}+a_{1}\mathbbm{1}_{[1]}$, the topological pressure
  $$P(X_\eta,\phi)=a_{00}(1-2d)+d\log(2^{a_{1}+a_{01}}+2^{2a_{00}}).$$
\end{proposition}

\begin{proof}
  For $n\in\N$, since $X_\eta=X_\B$,
  $$\mathcal{L}_n:=\mathcal{L}_n(X_\eta)=\{W\in\{0,1\}^n: W \text{ is $\B$-admissible}\}.$$
  For $K\in\N$, let
  $$\mathcal{L}_{n,K}=\{W\in\{0,1\}^n: W \text{ is $\{b_1,b_2,\cdots,b_K\}$-admissible}\}.$$
  So $\mathcal{L}_n\subset\mathcal{L}_{n,K}$ for any $K\ge 1$. Let $N_{n,K}:=nb_1b_2\cdots b_K$. Similar with the proof of Proposition K in \cite{Dym18}, we can obtain $W\in\mathcal{L}_{N_{n,K},K}$ by the following ways:\\
  (a)choose any $(r_1,\cdots, r_K)\in\prod_{k=1}^K\Z/b_k\Z$. Then for $j\in\bigcup_{k=1}^K(r_k+b_k\Z)$, set $W[j]=0$ when $j\in[0, N_{n,K}-1]$;\\
  (b)for $j\in[0, N_{n,K}-1]\setminus\bigcup_{k=1}^K(r_k+b_k\Z)$, complete the word $W$ by choosing arbitrarily $W[i]\in\{0,1\}$.\\
  So $\#_1W$ is ranged over $0$ to $N_{n,K}(1-d(\bigcup_{k=1}^K(r_k+b_k\Z)))$. By Lemma \ref{dK}, $N_{n,K}(1-d(\bigcup_{k=1}^K(r_k+b_k\Z)))\le N_{n,K}(1-d_K)$ where $d_K:=d(\mathcal{M}_{\{b_1,\cdots, b_K\}})$.

  Fixed any $\epsilon>0$, by Theorem \ref{ddK}, choose large enough $K$, such that $1-d-\epsilon<d_K<1-d+\epsilon$(notice that $d=d(\mathcal{F}_\B)=1-d(\mathcal{M}_\B)$). Because $2\in\B$, the word $11$ does not appear in any $x\in X_\eta$.

We claim that, for any $W\in\mathcal{L}_{N_{n,K},K}$,
$$\sup_{x\in W}\sum_{i=0}^{N_{n,K}-1}\phi(S^ix)\le a_{1}\#_1W+a_{01}\#_1W+a_{00}(N_{n,K}-2\#_1W)+\mathrm{Var}\phi([0]).$$
For $W\in\mathcal{L}_{N_{n,K},K}$, there are three cases:\\
(1)$W[N_{n,K}-1]=1$: Because $b_1=2$, $N_{n,K}$ is even, which implies that $W[0]=0$. So
$\#\{i\in[0,N_{n,K}-1]: W[i,i+1]=[01]\}=\#_1W$ and
$$
\begin{aligned}
        &\sup_{x\in W}\sum_{i=0}^{N_{n,K}-1}\phi(S^ix)\\
    =&a_{1}\#_1W+a_{01}\#_1W+a_{00}(N_{n,K}-2\#_1W)\\
    \le &a_{1}\#_1W+a_{01}\#_1W+a_{00}(N_{n,K}-2\#_1W)+\mathrm{Var}\phi([0]);
\end{aligned}
$$\\
(2)$W[0]=W[N_{n,K}-1]=0$: Because $W[0]=0$, we also have $\#\{i\in[0,N_{n,K}-1]: W[i,i+1]=[01]\}=\#_1W$. So
$$
\begin{aligned}
        &\sup_{x\in W}\sum_{i=0}^{N_{n,K}-1}\phi(S^ix)\\
    \le & a_{1}\#_1W+a_{01}\#_1W+a_{00}(N_{n,K}-1-2\#_1W)+\sup\phi([0])\\
    \le &a_{1}\#_1W+a_{01}\#_1W+a_{00}(N_{n,K}-2\#_1W)+\mathrm{Var}\phi([0]);
\end{aligned}
$$\\
(3)$W[0]=1$ and $W[N_{n,K}-1]=0$: The quantity $\#\{i\in[0,N_{n,K}-1]: W[i,i+1]=[01]\}=\#_1W-1$. So
$$
\begin{aligned}
        &\sup_{x\in W}\sum_{i=0}^{N_{n,K}-1}\phi(S^ix)\\
    \le & a_{1}\#_1W+a_{01}(\#_1W-1)+a_{00}(N_{n,K}-2\#_1W)+\sup\phi([0])\\
    \le &a_{1}\#_1W+a_{01}\#_1W+a_{00}(N_{n,K}-2\#_1W)+\mathrm{Var}\phi([0]).
\end{aligned}
$$\\
It ends the proof of the claim.

So
  $$
  \begin{aligned}
    &Z_{N_{n,K}}(X_\eta,\phi)\\
    \le &\sum_{W\in\mathcal{L}_{n,K}}2^{\sup_{x\in W}\sum_{i=0}^{N_{n,K}-1}\phi(S^ix)}\\
    \le & \sum_{W\in\mathcal{L}_{n,K}}2^{a_{1}\#_1W+a_{01}\#_1W+a_{00}(N_{n,K}-2\#_1W)+\mathrm{Var}\phi([0])}\\
    \le & \prod_{k=1}^Kb_k\cdot\sum_{i=0}^{N_{n,K}(1-d_K)}\binom{N_{n,K}(1-d_K)}{i}2^{a_{1}i+a_{01}i+a_{00}(N_{n,K}-2i)+\mathrm{Var}\phi([0])}\\
    =&\prod_{k=1}^Kb_k\cdot2^{a_{00}N_{n,K}+\mathrm{Var}\phi([0])-2a_{00}N_{n,K}(1-d_K)}(2^{a_{1}+a_{01}}+2^{2a_{00}})^{N_{n,K}(1-d_K)}\\
    =&\prod_{k=1}^Kb_k\cdot2^{a_{00}N_{n,K}(2d_K-1)+\mathrm{Var}\phi([0])}(2^{a_{1}+a_{01}}+2^{2a_{00}})^{N_{n,K}(1-d_K)}.
  \end{aligned}
  $$
  Thus,
  $$
  \begin{aligned}
    P(X_\eta,\phi)
    &\le a_{00}(2d_K-1)+(1-d_K)\log(2^{a_{1}+a_{01}}+2^{2a_{00}})\\
    &\le a_{00}(1-2d+2\epsilon)+(d+\epsilon)\log(2^{a_{1}+a_{01}}+2^{2a_{00}}),
  \end{aligned}
  $$
  which shows that $P(X_\eta,\phi)\le a_{00}(1-2d)+d\log(2^{a_{1}+a_{01}}+2^{2a_{00}})$ by the arbitrariness of $\epsilon$.

  To complete the proof , it remains to show that the inverse inequality. For any $n\in\N$, let
  $$p(n):=\#([1,n]\cap\mathcal{F}_\B).$$
  The set
  $$\{W\in\{0,1\}^n:W\le \eta[1,n]\}=\prod_{i\in\Z\cap[1,n]\setminus\mathcal{F}_\B}\{0\}\times\prod_{i\in[1,n]\cap\mathcal{F}_\B}\{0,1\}\subset\mathcal{L}_n.$$
  Thus
  $$
  \begin{aligned}
    Z_n(X_\eta,\phi)
    &\ge \sum_{i=0}^{p(n)}\binom{p(n)}{i}2^{a_{1}i+a_{01}i+a_{00}(n-2i)-2|\phi|}\\
    &=2^{a_{00}n-2a_{00}p(n)-2|\phi|}(2^{a_{1}+a_{01}}+2^{2a_{00}})^{p(n)},
  \end{aligned}
  $$
  where $|\phi|=\sup_{x\in X_\eta}|\phi(x)|$.
  So
  $$
  \begin{aligned}
    P(X_\eta,\phi)
    &\ge \lim_{n\to\infty}a_{00}\left(1-\frac{2p(n)}{n}\right)-\frac{2|\phi|}{n}+\frac{p(n)}{n}\log(2^{a_{1}+a_{01}}+2^{2a_{00}})\\
    &=a_{00}(1-2d)+d\log(2^{a_{1}+a_{01}}+2^{2a_{00}}),
  \end{aligned}
  $$
  which ends the proof.
\end{proof}

\begin{remark}\label{r:5.7}
	For $n\ge 1$, let $\mathscr{C}_n=\{\sum_{W\in\mathcal{L}_n(X_\eta)}a_W\mathbbm{1}_W: a_W\in\R,W\in\mathcal{L}_n(X_\eta)\}$. We also consider the function $\phi\in\mathscr{C}_n\setminus\mathscr{C}_2$ for $n\ge 3$.
	But in the calculation of topological pressure on $X_\eta$, it is not easy to estimate the frequency of the $n$-length word $A$ with $\#_1A\ge 2$ appearing in $W\in\mathcal{L}(X_\eta)$.
	This difficulty arises for $X_\eta$ with $2\in\B$.
	Also, for $\phi\in\mathscr{C}_2$ and $X_\eta$ with $2\notin\B$, this difficulty will arise because the word $11$ will appear in some $W\in\mathcal{L}(X_\eta)$.
	So in such cases, it is not easy to calculate or estimate the topological pressure for $\phi$.
\end{remark}

The measure entropy $h_{\nu_\eta*B_{q,1-q}}(X_\eta,S)$ is given in \cite{Prz15}.
\begin{proposition}[Proposition 2.1.9 \cite{Prz15}]
  If $\kappa=B_{p,1-p}$, then
  $$h_{\nu_\eta*\kappa}(X_\eta,S)=(-p\log p-(1-p)\log(1-p))\prod_{i=1}^\infty\left(1-\frac{1}{b_i}\right).$$
\end{proposition}

The next proposition shows that for some $p$, $\nu_\eta*B_{p,1-p}$ is an equilibrium state for $\phi=a_{00}\mathbbm{1}_{[00]}+a_{01}\mathbbm{1}_{[01]}+a_{1}\mathbbm{1}_{[1]}$.

\begin{proposition}\label{p:EqState}
  Suppose that $\B=\{b_1,b_2,\cdots\}$ satisfies (\ref{condition1}) and $b_1=2$. For $\phi=a_{00}\mathbbm{1}_{[00]}+a_{01}\mathbbm{1}_{[01]}+a_{1}\mathbbm{1}_{[1]}$, $\nu_\eta*B_{p,1-p}$ is an equilibrium state for $\phi$ where
  $$p=\frac{2^{2a_{00}}}{2^{a_1+a_{01}}+2^{2a_{00}}}.$$
\end{proposition}

\begin{proof}
  Since $2\in\B$, the word $11$ does not appear in $\eta$, which implies that $[11]\cap X_\eta=\emptyset$. For $[00]$, $[01]$ and $[10]$, their measures are given as follows:
  \begin{equation}\label{nu00}
  \begin{aligned}
    \nu_\eta([00])=&\lim_{n\to\infty}\frac{\#\{i\in[0,n-1]: S^i\eta\in[0]\text{ and }S^{i+1}\eta\in[0]\}}{n}\\
    =&\lim_{n\to\infty}\frac{n-2\#([0,n-1]\cap\mathcal{F}_\B)}{n}\\
    =&1-2d;
  \end{aligned}
  \end{equation}
  \begin{equation}\label{nu01}
  \begin{aligned}
    \nu_\eta([01])=&\lim_{n\to\infty}\frac{\#\{i\in[0,n-1]: S^{i+1}\eta\in[1]\}}{n}\\
    =&\lim_{n\to\infty}\frac{\#([1,n]\cap\mathcal{F}_\B)}{n}\\
    =&d;
  \end{aligned}
  \end{equation}
  \begin{equation}\label{nu10}
  \begin{aligned}
    \nu_\eta([10])=&\lim_{n\to\infty}\frac{\#\{i\in[0,n-1]: S^i\eta\in[1]\}}{n}\\
    =&\lim_{n\to\infty}\frac{\#([0,n-1]\cap\mathcal{F}_\B)}{n}\\
    =&d.
  \end{aligned}
  \end{equation}
  Those three equations follow from the fact that $11$ does not appear in $\eta$. Thus, by the Lemma \ref{JMlem1.1} and equations (\ref{nu00}), (\ref{nu01}) and (\ref{nu10}),
  \begin{equation}\label{kappa}
  \begin{aligned}
    \nu_\eta*B_{p,1-p}([00])=&\nu_\eta([00])+\nu_\eta([01])p+\nu_\eta([10])p\\
    =&1-2d+2dp,\\
    \nu_\eta*B_{p,1-p}([01])=&\nu_\eta([01])(1-p)=d(1-p),\\
    \nu_\eta*B_{p,1-p}([10])=&\nu_\eta([10])(1-p)=d(1-p).\\
  \end{aligned}
  \end{equation}

  So
  $$\int\phi d\nu_\eta*B_{p,1-p}=a_{00}(1-2d)+d(2a_{00}p+(a_{01}+a_{1})(1-p)).$$

  Notice that
  $$\log(2^{a_{1}+a_{01}}+2^{2a_{00}})=(-p\log p-(1-p)\log(1-p))+2a_{00}p+(a_{01}+a_{1})(1-p)$$
  if and only if
  $$p=\frac{2^{2a_{00}}}{2^{a_1+a_{01}}+2^{2a_{00}}}.$$
  It is showed before that $d=\prod_{i=1}^\infty(1-1/b_i)$. So when $$p=\frac{2^{2a_{00}}}{2^{a_1+a_{01}}+2^{2a_{00}}},$$
  we have
  $$P(X_\eta,\phi)=h_{\nu_\eta*B_{p,1-p}}(X_\eta,S)+\int\phi d\nu_\eta*B_{p,1-p},$$
  which implies that $\nu_\eta*B_{p,1-p}$ is the equilibrium state for $\phi$.
\end{proof}

Next, we will prove the uniqueness of equilibrium state. In \cite{Pec15} and \cite{Prz15}, they prove intrinsic ergodicity of the squarefree flow and $\B$-free system. We mainly use their methods to prove the uniqueness of the equilibrium state for $\phi=a_{00}\mathbbm{1}_{[00]}+a_{01}\mathbbm{1}_{[01]}+a_{1}\mathbbm{1}_{[1]}$.

Let $I=(i_1,i_2,\dots)$ where $i_k\in\{1,2,\dots,b_k-1\}$ for each $k\ge 1$. Define
$$X_I=\{x\in X_\eta:\text{ for any }k\ge 1,\,|\supp{x}\text{ mod }b_k|=b_k-i_k\},$$
and for any $K\ge 1$,
$$d_{I,K}=\prod_{k=1}^K\left(1-\frac{i_k}{b_k}\right),\quad d_I=\prod_{k=1}^\infty\left(1-\frac{i_k}{b_k}\right).$$
For convenience, we set $X_1:=X_{(1,1,\dots)}$. Notice that $X_1$ is Borel and $SX_1=X_1$.

\begin{lemma}
	Suppose that $\B=\{b_1,b_2,\cdots\}$ satisfies (\ref{condition1}) and $b_1=2$. Let $I=(i_1,i_2,\dots)$ where $i_k\in\{1,2,\dots,b_k\}$ for each $k\ge 1$. For $\phi=a_{00}\mathbbm{1}_{[00]}+a_{01}\mathbbm{1}_{[01]}+a_{1}\mathbbm{1}_{[1]}$, we have
	$$P(\overline{X}_I,S)\le a_{00}(1-2d_I)+d_I\log(2^{2a_{00}}+2^{a_{01}+a_1}),$$
	where $\overline{X}_I$ is the closure of $X_I$.
\end{lemma}

\begin{proof}
	For $n\in\N$, notice that
	$$\mathcal{L}_n(\overline{X}_I)=\{W\in\{0,1\}^n: |\supp{W}\text{ mod }b_k|\le b_k-i_k\text{ for each }k\ge 1\}.$$
	For $K\in\N$, let
	$$\mathcal{L}_{n,I,K}=\{W\in\{0,1\}^n: |\supp{W}\text{ mod }b_k|\le b_k-i_k\text{ for each }k=1,2,\dots,K\}.$$
	So $\mathcal{L}_n(\overline{X}_I)\subset\mathcal{L}_{n,I,K}$ for any $K\ge 1$. Let $N_{n,K}:=nb_1b_2\cdots b_K$.
	
	We can obtain $W\in\mathcal{L}_{N_{n,K},I,K}$ by the following ways:\\
	(a)choose any $(Z_1,\dots, Z_K)$ with $Z_k\subset \Z/b_k\Z$ and $\#Z_k=i_k$ for each $k=1,\dots,K$. Then let
	$$Z(Z_1,\dots,Z_K)=\{0\le j<N_{n,K}:\text{for any }k=1,2,\dots,K,\,j\text{ mod }b_k\in Z_k\},$$
	and for $j\in Z(Z_1,\dots,Z_K)$, set $W[j]=0$;\\
	(b)for $j\in[0, N_{n,K}-1]\setminus Z(Z_1,\dots,Z_K)$, complete the word $W$ by choosing arbitrarily $W[j]\in\{0,1\}$.\\
	Since $\#Z(Z_1,\dots,Z_K)=N_{n,K}(1-d_{I,K})$, we have $\#_1W$ is ranged over $0$ to $N_{n,K}d_{I,K}$. And there are at most $\binom{b_1}{i_1}\cdots\binom{b_K}{i_K}$ choices of $(Z_1,\dots, Z_K)$ in (a).
	Similar with the claim in the proof of Proposition \ref{PonX}, we have
	$$\sup_{x\in W}\sum_{i=0}^{N_{n,K}-1}\phi(S^ix)\le a_{1}\#_1W+a_{01}\#_1W+a_{00}(N_{n,K}-2\#_1W)+\mathrm{Var}\phi([0])$$
	for each $W\in\mathcal{L}_{N_{n,K},I,K}$.
	So
	$$
	\begin{aligned}
	&Z_{N_{n,K}}(\overline{X}_I,\phi)\\
	\le &\sum_{W\in\mathcal{L}_{N_{n,K},I,K}}2^{\sup_{x\in W}\sum_{i=0}^{N_{n,K}-1}\phi(S^ix)}\\
	\le & \sum_{W\in\mathcal{L}_{N_{n,K},I,K}}2^{a_{1}\#_1W+a_{01}\#_1W+a_{00}(N_{n,K}-2\#_1W)+\mathrm{Var}\phi([0])}\\
	\le & \prod_{k=1}^K\binom{b_k}{i_k}\cdot\sum_{i=0}^{N_{n,K}d_{I,K}}\binom{N_{n,K}d_{I,K}}{i}2^{a_{1}i+a_{01}i+a_{00}(N_{n,K}-2i)+\mathrm{Var}\phi([0])}\\
	=&\prod_{k=1}^K\binom{b_k}{i_k}\cdot2^{a_{00}N_{n,K}+\mathrm{Var}\phi([0])-2a_{00}N_{n,K}d_{I,K}}(2^{a_{1}+a_{01}}+2^{2a_{00}})^{N_{n,K}d_{I,K}}\\
	=&\prod_{k=1}^K\binom{b_k}{i_k}\cdot2^{a_{00}N_{n,K}(1-2d_{I,K})+\mathrm{Var}\phi([0])}(2^{a_{1}+a_{01}}+2^{2a_{00}})^{N_{n,K}d_{I,K}}.
	\end{aligned}
	$$
	Thus
	$$P(\overline{X}_I,\phi)\le a_{00}(1-2d_{I,K})+d_{I,K}\log(2^{a_{1}+a_{01}}+2^{2a_{00}}),$$
	and let $K\To\infty$, which ends the proof.
\end{proof}
The next lemma shows that we can use the methods in the proof of intrinsic ergodicity in \cite{Pec15} and \cite{Prz15}.
\begin{lemma}\label{l:X1}
	Suppose that $\B=\{b_1,b_2,\cdots\}$ satisfies (\ref{condition1}) and $b_1=2$. Let $\mu\in\mathcal{M}^e(X_\eta,S)$ be an equilibrium state for $\phi=a_{00}\mathbbm{1}_{[00]}+a_{01}\mathbbm{1}_{[01]}+a_{1}\mathbbm{1}_{[1]}$. Then $\mu(X_1)=1$.
\end{lemma}

\begin{proof}
	First, there is unique $I=(i_1,i_2,\dots)\in\prod_{k\ge 1}\{1,2,\dots,b_k-1\}$ such that $\mu(X_I)=1$, which can be proved by a similar method in \cite[Lemma 3.3]{Pec15} although it is the case of $\mathscr{B}=\{p^2:p\text{ is prime number}\}$.
	Thus
	$$
	\begin{aligned}
	P(X_\eta,\phi)=&h_\mu(X_\eta,S)+\int_{X_\eta}\phi d\mu\\
	=&P(\overline{X}_I,\phi)\le a_{00}(1-2d_I)+d_I\log(2^{a_{1}+a_{01}}+2^{2a_{00}}).
	\end{aligned}
	$$
	Since $P(X_\eta,\phi)=a_{00}(1-2d)+d\log(2^{a_{1}+a_{01}}+2^{2a_{00}})$, we have $d_I\ge d$. When $\B$ satisfies the condition (\ref{condition1}), it follows by \cite{Abd13} that $d>0$, which implies that $i_k=1$ for each $k\ge 1$ and $X_I=X_1$.
\end{proof}

By this lemma, we can use the methods in \cite{Abd13} and \cite{Prz15} to prove the uniqueness of equilibrium state. Given $k\ge 1$ and $z\in\Z/b_k\Z$, set
$$
\begin{aligned}
\Omega_{k,z}=&\{\omega\in\Omega:\omega(k)=z\},\\
E_{k,z}=&\{\omega\in\Omega:\text{for any }s\ge 1,\,\varphi(\omega)(-z+sb_k)=0\}
\end{aligned}
$$
and
$$\Omega'_0=\bigcap_{k\ge 1}\bigcap_{z\in\Z/b_k\Z}\left(E^c_{k,z}\cup\Omega_{k,z}\right),\,
\Omega_0=\bigcap_{k\in\Z}T^k\Omega'_0.$$

\begin{lemma}[Proposition 3.2 in \cite{Abd13}]
	We have $\mathbb{P}(\Omega_0)=1$ and $\varphi|_{\Omega_0}$ is 1-1.
\end{lemma}

Define a Borel map $\theta:X_1\To\Omega$ (cf.\cite{Pec15}) satisfy that
$$-\theta(y)(i)\notin\supp{y}\text{ mod }b_i\text{ for all }i\ge 1.$$
Since $|\supp{y}\text{ mod }b_i|=b_i-1$ for each $y\in X_1$ and each $i\ge 1$, the map $\theta$ is well defined.

\begin{lemma}[Lemma 2.5 in \cite{Prz15}]\label{l:5.13}
	We have:
	\begin{itemize}
		\item[(i)] $T\circ\theta=\theta\circ S$;
		\item[(ii)] for each $y\in X_1$, $y\le \varphi(\theta(y))$;
		\item[(iii)] $\varphi(\Omega_0)\subset X_1$ (In particular, $\theta\circ\varphi|_{\Omega_0}=id_{\Omega_0}$).
	\end{itemize}
\end{lemma}

Fix a measure $\mu$ on $X_1$ is the equilibrium state for $\phi=a_{00}\mathbbm{1}_{[00]}+a_{01}\mathbbm{1}_{[01]}+a_{1}\mathbbm{1}_{[1]}$.

\begin{lemma}\label{l:5.14}
	We have $\theta_*(\mu)=\mathbb{P}$.
\end{lemma}

\begin{proof}
	By Lemma \ref{l:5.13} and the fact that $(\Omega, T)$ is uniquely ergodic, it ends the proof.
\end{proof}

Let $Y:=\theta^{-1}(\Omega_0)$. By Lemma \ref{l:5.13} and Lemma \ref{l:5.14}, we have $\mu(Y)=\theta_*(\mu)(\Omega_0)=\mathbb{P}(\Omega_0)=1$.

Here, we recall some observations in \cite[Section 2.2]{Prz15}. Let $Q=\{Q_0=[0]\cap Y,Q_1=[1]\cap Y\}$ be the generating partition of $Y$. Set
$$Q^-:=\bigvee_{j\ge 1}S^{-j}Q,\text{ and }\mathcal{A}:=\theta^{-1}(\mathcal{B}(\Omega)),$$
where $\mathcal{B}(\cdot)$ stands for the Borel $\sigma$-algebra. Since $Q$ is a generating partition, the $\sigma$-algebra $\cap_{m\ge 0}S^{-m}Q^-$ is the Pinsker $\sigma$-algebra of $(Y,\mathcal{B}(Y),\mu,S)$. By \cite[Lemma 2.11]{Prz15}, $\mathcal{A}\subset \cap_{m\ge 0}S^{-m}Q^-$ modulo $\mu$. It follows that almost every atom of the partition corresponding to the Pinsker $\sigma$-algebra of $(Y,\mathcal{B}(Y),\mu,S)$ is contained in an atom of the partition of $Y$ corresponding to $\mathcal{A}$. Also, we have $\mathcal{A}\subset S^{-m}Q^-$.

Fix $m\ge 0$. Let $\pi_m$ be the quotient map from $Y$ to the quotient space $Y/S^{-m}Q^-$. Let $\bar{\mu}_m:=(\pi_m)_*(\mu)$. So $S$ acts naturally on the quotient space $Y/S^{-m}Q^-$ as an endomorphism preserving $\bar{\mu}_m$ and $\pi_m\circ S=S\circ \pi_m$. Also, it can define the quotient map $\rho_m:Y/S^{-m}Q^-\To\Omega$ with $\rho_m\circ S=T\circ \rho_m$. Then $(\rho_m)_*(\bar{\mu}_m)=\mathbb{P}$. Thus it follows that $S^k\circ\varphi\circ\rho_m=\varphi\circ\rho_m\circ S^k$, that is,
\begin{equation}\label{e:commute}
\varphi\circ\rho_m(\bar y)(m+k)=\varphi\circ\rho_m(S^k\bar y)(m)\text{ for any }k\in\Z.
\end{equation}

We identify points in $Y/S^{-m}Q^-$ by the following ways: for $y\in Y$, let $\bar y$ be the atom of the partition associated to $S^{-m}Q^-$ which contains $y$, that is,
$$\bar y=\cdots i_{-1}i_0\Longleftrightarrow y\in S^{-m-1}Q_{i_0}\cap S^{-m-2}Q_{i_{-1}}\cap \cdots.$$

\begin{lemma}[Lemma 2.13 in \cite{Prz15}]\label{l:5.15}
	For each $m\ge 0$, $r=0,1,\dots,2m$ and $\bar{\mu}_m$-a.e. $\bar y\in Y/S^{-m}Q^-$, we have
	$$
	\begin{aligned}
	&\bar\mu_m(S^{m-r}Q_{i_{m-r}}|S^{m-r-1}Q_{i_{m-r-1}}\cap\cdots\cap S^{-m+1}Q_{i_{-m+1}}\cap S^{-m}Q_{i_{-m}}\cap S^{-m}Q^-)(\bar y)\\
	=&\bar\mu_m(S^{-m}Q_{i_{m-r}}|S^{-m}Q^-)(\bar yi_{-m}\dots i_{m-r-1})
	\end{aligned}
	$$
	for each choice of $i_k\in\{0,1\}$, $-m\le k\le m$.
\end{lemma}

Now, we can prove the uniqueness of equilibrium state for $\phi=a_{00}\mathbbm{1}_{[00]}+a_{01}\mathbbm{1}_{[01]}+a_{1}\mathbbm{1}_{[1]}$ in the $\B$-free system with $2\in\B$.

\begin{theorem}\label{t:unique}
	Suppose that $\B=\{b_1,b_2,\cdots\}$ satisfies (\ref{condition1}) and $b_1=2$. Let $\phi=a_{00}\mathbbm{1}_{[00]}+a_{01}\mathbbm{1}_{[01]}+a_{1}\mathbbm{1}_{[1]}$. Then $\phi$ has unique equilibrium state.
\end{theorem}

\begin{proof}
	Let $\mu$ be an ergodic equilibrium state for $\phi$.
	We will show that the conditional measures $\mu_\omega$ in the disintegration
	$$\mu=\int_\Omega\mu_\omega d\mathbb{P}$$
	of $\mu$ over $\mathbb{P}$ given by the mapping $\theta:X_1\To \Omega$ are unique $\mathbb{P}$-a.e. $\omega\in\Omega_0$.
	This will show the uniqueness of equilibrium state.
	We define another measure $\mu^*$ in the following way, which will be showed that $\mu=\mu^*$.
	For each $\omega\in\Omega_0$, we have $\varphi(\omega)\in Y$.
	By Lemma \ref{l:5.13}, $\varphi(\omega)$ is the largest element in $\theta^{-1}(\omega)$.
	Thus for each $u=u_{-k}\cdots u_k\in\{0,1\}^{2k+1}\le \varphi(\omega)[-k,k]$, we set $$\mu^*_\omega([u]):=\prod_{-k\le i\le k,\varphi(\omega)(i)=1}\lambda^1_{u_i},$$
	where $\lambda^1_0=\frac{2^{2a_{00}}}{2^{2a_{00}}+2^{a_{01}+a_1}}$ and $\lambda^1_1=1-\lambda^1_0$.
	Finally, we set
	$$\mu^*=\int_{\Omega_0}\mu^*_\omega d\mathbb{P}.$$
	
	We will show that $\mu_\omega=\mu^*_\omega$ for $\mathbb{P}$-a.e. $\omega\in\Omega_0$, which implies that $\mu=\mu^*$. We will prove that for any $m\ge 0$ and any  $A\in\bigvee_{j=-m}^mS^jQ$,
	\begin{equation}\label{e:unique}
		\mu_\omega(A)=\mu^*_\omega(A),\,\mathbb{P}\text{-a.e. }\omega\in\Omega.
	\end{equation}
	Recall that
	\begin{equation}\label{e:tool1}
		\mu_\omega(A)=\mathbb{E}^\mu(A|\Omega)(\omega).
	\end{equation}
	To get the equation (\ref{e:unique}), we will step by step make use of the equality
	\begin{equation}\label{e:tool2}
		\mathbb{E}^\mu(A|\Omega)(\omega)=\mathbb{E}^\mu(\mathbb{E}^\mu(A|Y/S^{-m}Q^-)(\bar y_m)|\Omega)(\omega)
	\end{equation}
	where $A\in\bigvee_{j=-m}^mS^iQ$, $m\ge 0$ and show that
	\begin{equation}\label{e:tool3}
		\mathbb{E}^\mu(A|Y/S^{-m}Q^-)(\bar y_m)=\mu^*_\omega(A)
	\end{equation}
	for all $\bar y_m$ having the same $\rho_m$-projection $\omega$.
	
	First, we need some denotations. For $m\ge 0$, let
	$$
	\begin{aligned}
	\hat C^{00}_m:=&\varphi^{-1}(S^{-m}(Q_0\cap S^{-1}Q_0))=\{\omega\in\Omega_0:\varphi(\omega)(m)=0,\varphi(\omega)(m+1)=0\},\\
	\hat C^{01}_m:=&\varphi^{-1}(S^{-m}(Q_0\cap S^{-1}Q_1))=\{\omega\in\Omega_0:\varphi(\omega)(m)=0,\varphi(\omega)(m+1)=1\},\\
	\hat C^1_m:=&\varphi^{-1}(S^{-m}Q_1)=\{\omega\in\Omega_0:\varphi(\omega)(m)=1\}.\\
	\end{aligned}
	$$
	Then $\Omega_0=\hat C^{00}_m\cup\hat C^{01}_m\cup\hat C^1_m$ and $Y=\theta^{-1}(\Omega_0)=\theta^{-1}(\hat C^{00}_m)\cup\theta^{-1}(\hat C^{01}_m)\cup\theta^{-1}(\hat C^1_m)$.
	Let $B_m^j:=\rho_m^{-1}(\hat C^j_m)$ for $j\in\{00,01,1\}$.
	Then we have
	$$Y/S^{-m}Q^-=B_m^{00}\cup B_m^{01}\cup B_m^1.$$
	By the definition of $\rho_m$,
	$$\bar\mu_m(B_m^{00})=\mu(\theta^{-1}(\hat C^{00}_m))=\mathbb{P}(\hat C^{00}_m)=\nu_\eta([00])=1-2d,$$
	and $\bar\mu_m(B_m^{01})=\bar\mu_m(B_m^1)=\nu_\eta([01])=\nu_\eta([1])=d$.
	
	Now, we prove the equality (\ref{e:unique}) in two cases: $m=0$ and $m\ge 0$, which the first case is not necessary but it can be seen as a toy model for the second case.
	\begin{itemize}
		\item[(1)] Toy model: the case of $m=0$.
	\end{itemize}
	We first show that $\theta^{-1}(\hat C^{00}_0)\subset Q_0\cap S^{-1}Q_0$.
	For any $y\in\theta^{-1}(\hat C^{00}_0)$, we have $\varphi(\theta(y))(0)=\varphi(\theta(y))(1)=0$.
	Since $y\le \varphi(\theta(y))$, we have $y\in Q_0\cap S^{-1}Q_0$.
	Thus for any $\bar y\in B^{00}_0$, $\pi^{-1}_0(\bar y)\subset \pi^{-1}_0(B^{00}_0)=\theta^{-1}(\hat C^{00}_0)\subset Q_0\cap S^{-1}Q_0$, which implies that for $\bar y\in B^{00}_0$,
	$$
	\begin{aligned}
	&\bar \mu_0(Q_0\cap S^{-1}Q_0|Q^-)(\bar y)=1,\\
	&\bar \mu_0(Q_0\cap S^{-1}Q_1|Q^-)(\bar y)=0,\\
	&\bar \mu_0(Q_1|Q^-)(\bar y)=0.
	\end{aligned}
	$$
	Therefore,
	\begin{equation}\label{e:B^{00}_0}
		\begin{aligned}
		&\int_{B^{00}_0}a_{00}\bar\mu_0(Q_0\cap S^{-1}Q_0|Q^-)(\bar y)+a_{01}\bar\mu_0(Q_0\cap S^{-1}Q_1|Q^-)(\bar y)\\
		&+a_1\bar \mu_0(Q_1|Q^-)(\bar y) d\bar\mu_0(\bar y)\\
		=&a_{00}\mu_0(B^{00}_0)=a_{00}(1-2d).
		\end{aligned}
	\end{equation}
	
	Also, we have $\theta^{-1}(\hat C^{00}_0\cup\hat C^{01}_0)\subset Q_0$.
	Thus for any $\bar y\in B^{00}_0\cup B^{01}_0$,
	\begin{equation}\label{e:main1}
	(\bar \mu_0(Q_0|Q^-)(\bar y),\bar \mu_0(Q_1|Q^-)(\bar y))=(1,0)=:(\lambda^0_0,\lambda^0_1),
	\end{equation}
	which implies that $H_\mu(Q|Q^-)(\bar y)=0$.
	In particular, for $\bar y\in B^{01}_0$, we have $$\bar\mu_0(Q_0\cap S^{-1}Q_j|Q^-)(\bar y)=\bar\mu_0(S^{-1}Q_j|Q^-)(\bar y)\text{ for }j=0,1.$$
	
	We claim that $SB^{01}_0=B^1_0$.
	Indeed, for any $\bar y\in B^{01}_0$, $\varphi(\rho_0(S\bar y))=S\varphi(\rho_0(\bar y))\in S\varphi(\hat C^{01}_0)\subset [1]$. Thus $S\bar y\in B^1_0$.
	Conversely, for any $\bar y\in B^1_0$, let $y$ with $\pi_0(y)=\bar y$ and $\bar y'=\pi_0(S^{-1}y)$.
	So $S\bar y'=\bar y$.
	Then $S\varphi(\rho_0(\bar y'))=\varphi(\theta(y))\in[1]$.
	Since $b_1=2$, we have $\varphi(\rho_0(\bar y'))(0)=0$, $\varphi(\rho_0(\bar y'))(1)=1$.
	Thus $\bar y'\in B^{01}_0$, which ends the proof of the claim.
	Therefore,
	\begin{equation}\label{e:B^{01}_0}
	\begin{aligned}
	&\int_{B^{01}_0}a_{00}\bar\mu_0(Q_0\cap S^{-1}Q_0|Q^-)(\bar y)+a_{01}\bar\mu_0(Q_0\cap S^{-1}Q_1|Q^-)(\bar y)\\
	&+a_1\bar \mu_0(Q_1|Q^-)(\bar y) d\bar\mu_0(\bar y)\\
	=&\int_{B^{01}_0}a_{00}\bar\mu_0(S^{-1}Q_0|Q^-)(\bar y)+a_{01}\bar\mu_0(S^{-1}Q_1|Q^-)(\bar y)d\bar\mu_0(\bar y)\\
	=&\int_{B^1_0}a_{00}\bar\mu_0(Q_0|Q^-)(\bar y)+a_{01}\bar\mu_0(Q_1|Q^-)(\bar y)d\bar\mu_0(\bar y).
	\end{aligned}
	\end{equation}
	
	Since $2\in\B$, one can prove that $\theta^{-1}(\hat C^1_0)\subset S^{-1}Q_0$.
	Indeed, if $y\in\theta^{-1}(\hat C^1_0)$, we have $\varphi(\theta(y))(0)=1$.
	Thus $y(1)\le \varphi(\theta(y))(1)=0$ since $2\in\B$ and $y\le \varphi(\theta(y))$.
	It follows that for $\bar y\in B^1_0$, we have $\pi_0^{-1}(\bar y)\subset S^{-1}Q_0$, that is, $\bar\mu(Q_0\cap S^{-1}Q_1|Q^-)(\bar y)=0$ and $\bar\mu_0(Q_0\cap S^{-1}Q_0|Q^-)(\bar y)=\bar\mu_0(Q_0|Q^-)(\bar y)$.
	Therefore,
	\begin{equation}\label{e:B^1_0}
	\begin{aligned}
	&\int_{B^1_0}a_{00}\bar\mu_0(Q_0\cap S^{-1}Q_0|Q^-)(\bar y)+a_{01}\bar\mu_0(Q_0\cap S^{-1}Q_1|Q^-)(\bar y)\\
	&+a_1\bar \mu_0(Q_1|Q^-)(\bar y) d\bar\mu_0(\bar y)\\
	=&\int_{B^1_0}a_{00}\bar\mu_0(Q_0|Q^-)(\bar y)+a_1\bar\mu_0(Q_1|Q^-)(\bar y)d\bar\mu_0(\bar y).
	\end{aligned}
	\end{equation}
	
	Sum up with (\ref{e:B^{00}_0}), (\ref{e:B^{01}_0}) and (\ref{e:B^1_0}), we have
	$$
	\begin{aligned}
	&P(X_\eta,\phi)
	=h_{\mu}(X_\eta,S)+\int\phi d\mu\\
	=&\int_{Y/Q^-}H_\mu(Q|Q^-)(\bar y)+a_{00}\bar\mu_0(Q_0\cap S^{-1}Q_0|Q^-)(\bar y)\\
	&+a_{01}\bar\mu_0(Q_0\cap S^{-1}Q_1|Q^-)(\bar y)+a_1\bar \mu_0(Q_1|Q^-)(\bar y) d\bar\mu_0(\bar y)\\
	=&a_{00}(1-2d)+\int_{B^1_0} H_\mu(Q|Q^-)(\bar y)+2a_{00}\bar\mu_0(Q_0|Q^-)(\bar y)\\
	&+(a_{01}+a_1)\bar\mu_0(Q_1|Q^-)(\bar y)d\bar\mu_0(\bar y)\\
	\le&a_{00}(1-2d)+\bar\mu_0(B^1_0)\log(2^{2a_{00}}+2^{a_{01}+a_1})\\
	=&a_{00}(1-2d)+d\log(2^{2a_{00}}+2^{a_{01}+a_1}).
	\end{aligned}
	$$
	The inequality comes from $\sum p_i(b_i-\log p_i)\le\log(\sum 2^{b_i})$ for any $\sum p_i=1$ and any $b_i$. So for $\bar\mu_0$-a.e. $\bar y\in B^1_0$, we have \begin{equation}\label{e:main2}
	(\bar\mu_0(Q_0|Q^-)(\bar y),\bar\mu_0(Q_1|Q^-)(\bar y))=(\lambda^1_0, \lambda^1_1).
	\end{equation}

	Notice that (\ref{e:main1}) and (\ref{e:main2}) do not depend on $\bar y$ itself but only on the values $\varphi(\rho_0(\bar y))(0)$ and $\varphi(\rho_0(\bar y))(1)$, which implies that (\ref{e:tool3}) holds. Sum up with (\ref{e:tool1}), (\ref{e:tool2}) and (\ref{e:tool3}), we conclude that in the disintegration of $\mu$ over $\mathbb{P}$ via $\theta$, for $\mathbb{P}$-a.e. $\omega\in\Omega$, $\mu_\omega(Q_j)=\mu_\omega^*(Q_j)$ for $j=0,1$.
	
	\begin{itemize}
		\item[(2)] General case: the case of $m\ge 0$.
	\end{itemize}
	Fix $m\ge 0$. As in the case of $m=0$, we obtain that $\theta^{-1}(\hat C^{00}_m)\subset S^{-m}(Q_0\cap S^{-1}Q_0)$, which implies that for $\bar y\in B^{00}_m$,
	$$
	\begin{aligned}
	&\bar \mu_m(S^{-m}(Q_0\cap S^{-1}Q_0)|S^{-m}Q^-)(\bar y)=1,\\
	&\bar \mu_m(S^{-m}(Q_0\cap S^{-1}Q_1)|S^{-m}Q^-)(\bar y)=0,\\
	&\bar \mu_m(S^{-m}Q_1|S^{-m}Q^-)(\bar y)=0.
	\end{aligned}
	$$
	Similar to the case of $m=0$, we have $\theta^{-1}(\hat C^{00}_m\cup\hat C^{01}_m)\subset S^{-m}Q_0$, which implies that for any $\bar y\in B^{00}_m\cup B^{01}_m$,
	\begin{equation}\label{e:main3}
	(\bar \mu_m(S^{-m}Q_0|S^{-m}Q^-)(\bar y),\bar \mu_m(S^{-m}Q_1|S^{-m}Q^-)(\bar y))=(1,0)=(\lambda^0_0,\lambda^0_1),
	\end{equation}
	and $H_\mu(S^{-m}Q|S^{-m}Q^-)(\bar y)=0$.
	In particular, for $\bar y\in B^{01}_m$, we have $$\bar\mu_m(S^{-m}(Q_0\cap S^{-1}Q_j)|S^{-m}Q^-)(\bar y)=\bar\mu_m(S^{-m}(S^{-1}Q_j)|S^{-m}Q^-)(\bar y)\text{ for }j=0,1.$$
	And we also obtain that $SB^{01}_m=B^1_m$ and $\theta^{-1}(\hat C^1_m)\subset S^{-m}(S^{-1}Q_0)$. Therefore, similar to the case of $m=0$, the computation of
	$$
	\begin{aligned}
	&h_{\mu}(X_\eta,S)+\int\phi\circ S^m d\mu\\
	=&\int_{Y/S^{-m}Q^-}H_\mu(S^{-m}Q|S^{-m}Q^-)+a_{00}\bar\mu_m(S^{-m}(Q_0\cap S^{-1}Q_0)|Q^-)(\bar y)\\
	&+a_{01}\bar\mu_m(S^{-m}(Q_0\cap S^{-1}Q_1)|S^{-m}Q^-)(\bar y)+a_1\bar \mu_m(S^{-m}Q_1|S^{-m}Q^-)(\bar y) d\bar\mu_m(\bar y)
	\end{aligned}
	$$
	leads to
	\begin{equation}\label{e:main4}
	(\bar\mu_m(S^{-m}Q_0|S^{-m}Q^-)(\bar y),\bar\mu_m(S^{-m}Q_1|S^{-m}Q^-)(\bar y))=(\lambda^1_0, \lambda^1_1).
	\end{equation}
	for $\bar\mu_m$-a.e. $\bar y\in B^1_m$.
	In order to prove that $\mu_\omega=\mu^*_\omega$ for $A\in\bigvee_{i=-m}^mS^iQ$, choose $(i_{-m},\dots,i_0,\dots,i_m)\in\{0,1\}^{2m+1}$.
	By the chain conditional probabilities and Lemma \ref{l:5.15}, we have
	$$
	\begin{aligned}
	&\bar\mu_m(\bigcap^{2m}_{r=0}S^{m-r}Q_{i_{m-r}}|S^{-m}Q^-)(\bar y)\\
	=&\prod^{2m}_{r=0}\bar\mu_m(S^{m-r}Q_{i_{m-r}}|S^{m-r-1}Q_{i_{m-r-1}}\cap\cdots\cap S^{-m}Q_{i_{-m}}\cap S^{-m}Q^-)(\bar y)\\
	=&\prod^{2m}_{r=0}\bar\mu_m(S^{-m}Q_{i_{m-r}}|S^{-m}Q^-)(\bar yi_{-m}\dots i_{m-r-1}).
	\end{aligned}
	$$
	By (\ref{e:main3}) and (\ref{e:main4}), for $\bar\mu$-a.e. $\bar y\in Y/S^{-m}Q^-$,
	$$\bar\mu_m(S^{-m}Q_{i_{m-r}}|S^{-m}Q^-)(\bar yi_{-m}\dots i_{m-r-1})=\lambda^{j_r}_{i_{m-r}},$$
	where $j_r=\varphi(\rho_m(\bar yi_{-m}\dots i_{m-r-1}))(m)$. And by equation (\ref{e:commute}), $j_r=\varphi(\rho_m(\bar y))(m+2m-r)$
	Sum up with (\ref{e:tool1}), (\ref{e:tool2}) and (\ref{e:tool3}), (\ref{e:unique}) is proved.
	
	It follows that $\mu_\omega=\mu^*_\omega$ for $\mathbb{P}$-a.e. $\omega\in\Omega$ and $\mu=\mu^*$, which ends the proof.
\end{proof}

\begin{remark}
	Similar to Remark \ref{r:5.7}, we do not know whether the uniqueness of equilibrium state holds for $\phi\in\mathscr{C}_n\setminus\mathscr{C}_2$ where $n\ge 3$.
\end{remark}


By using Theorem \ref{t:1.1}, we will show that $\nu_\eta*B_{p,1-p}$ is not Gibbs measure for some $\phi$.
We consider $\phi=a_{00}\mathbbm{1}_{[00]}+a_{01}\mathbbm{1}_{[01]}+a_{1}\mathbbm{1}_{[1]}$ with $a_1>\max\{a_{00},a_{01}\}$. It is necessary to sure that the condition
\begin{equation}\label{con1.1}
  P\le(\mathrm{Var}\phi([0])-\log (1-p))d+d^{\phi}-\mathrm{Var}\phi([0])
\end{equation}
can be satisfied, where $p=\frac{2^{2a_{00}}}{2^{a_1+a_{01}}+2^{2a_{00}}}$.
Firstly, we estimate the quantity $d^{\phi}$. It is showed that for $0<q<1$,
$$\int\phi d\nu_\eta*B_{q,1-q}=a_{00}(1-2d)+d(2a_{00}q+(a_{01}+a_{1})(1-q)).$$
So
\begin{equation}\label{dphi}
  d^{\phi}\ge \sup_{0<q<1}\int\phi d\nu_\eta*B_{q,1-q}=a_{00}(1-2d)+d\max\{2a_{00}, a_1+a_{01}\}.
\end{equation}

Since
$$P=a_{00}(1-2d)+d\log(2^{a_{1}+a_{01}}+2^{2a_{00}})=a_{00}(1-2d)+d(a_{1}+a_{01})-d\log(1-q),$$
we can replace the condition (\ref{con1.1}) by the condition
\begin{equation}\label{con1.2}
  d(a_1+a_{01})\le(d-1)\mathrm{Var}\phi([0])+d\max\{2a_{00}, a_1+a_{01}\}.
\end{equation}

Notice that $2\in\B$ implies that $0\le d<1/2$.
If $2a_{00}\le a_1+a_{01}$, then the condition (\ref{con1.2}) is satisfied when $\mathrm{Var}\phi([0])=0$, which means that $a_{00}=a_{01}$ and $2a_{00}\le a_1+a_{01}$ is natural.

If $2a_{00}> a_1+a_{01}$, then we have $a_1>a_{00}>a_{01}$. So the condition (\ref{con1.2}) becomes to be
$$(2d-1)(a_{00}-a_{01})+d(a_{00}-a_1)\ge 0,$$
which can not be satisfied when $a_1>a_{00}>a_{01}$.

So, by the above consideration, we have
\begin{proposition}
  Suppose that $\B=\{b_1,b_2,\cdots\}$ satisfies (\ref{condition1}) and $b_1=2$. For $\phi=a_0\mathbbm{1}_{[0]}+a_1\mathbbm{1}_{[1]}$, if $a_1\ge a_0$,
then the equilibrium state $\kappa=\nu_\eta*B_{p,1-p}$ is not Gibbs measure for $\phi$, where
  $$p=\frac{2^{a_0}}{2^{a_1}+2^{a_0}}.$$
\end{proposition}

\begin{proof}
  With the assumptions of $a_0$ and $a_1$, we have
  $$\sup\phi([1])\ge\sup\phi([0]),\text{ and }\mathrm{Var}\phi([1])=0\le\mathrm{Var}\phi([0])-\log (1-p).$$
  Since $a_1\ge a_0$, by inequality (\ref{dphi}),
  $$d^{\phi}\ge a_0(1-d)+a_1d.$$
  So we have
  $$
  \begin{aligned}
    P(X_\eta,\phi)= &a_0(1-d)+d\log(2^{a_1}+2^{a_0})\\
    =&a_0(1-d)+a_1d-d\log(1-p)\\
    \le&d^\phi-d\log (1-p).
  \end{aligned}
  $$
  Since $\nu_\eta$ is full support on $X_\eta$, $\kappa=\nu_\eta*B_{p,1-p}$ is full support on $X_\eta$. Therefore, $D_{\nu_\eta}=D$ and $D^\phi_\kappa=D^\phi$. Then by Theorem \ref{t:1.1}, $\kappa=\nu_\eta*B_{p,1-p}$ is not Gibbs measure for $\phi$.
\end{proof}

Here, we give an example that the equilibrium state $\nu_\eta*B_{p,1-p}$ is not Gibbs measure for more general $\phi=a_{00}\mathbbm{1}_{[00]}+a_{01}\mathbbm{1}_{[01]}+a_{1}\mathbbm{1}_{[1]}$ on $(X_\eta, S)$ with $2\in\B$, but we can not use Theorem \ref{t:1.1} directly.

\begin{proposition}\label{p:notGibbs}
  Suppose that $\B=\{b_1,b_2,\cdots\}$ satisfies (\ref{condition1}) and $b_1=2$. For $\phi=a_{00}\mathbbm{1}_{[00]}+a_{01}\mathbbm{1}_{[01]}+a_1\mathbbm{1}_{[1]}$, if $a_1\ge \max\{a_{00},a_{01}\}$ and $2a_{00}\le a_1+a_{01}$,
then the equilibrium state $\kappa=\nu_\eta*B_{p,1-p}$ is not Gibbs measure for $\phi$, where
  $$p=\frac{2^{2a_{00}}}{2^{a_1+a_{01}}+2^{2a_{00}}}.$$
\end{proposition}

\begin{proof}
  Since $\kappa$ is full support on $X_\eta$, $D^{\phi}_\kappa=D^{\phi}$. So for any $n\in\N$, there exists $x^{(n)}\in X_\eta$ such that
$$\sum_{i=0}^{n-1}\phi(S^ix^{(n)})
=\sup_{y\in X_\eta}\sum_{i=0}^{n-1}\phi(S^iy)
\ge nD^{\phi}.$$

Since $\nu_\eta$ is full support on $X_\eta$, $D_{\nu_\eta}=D$. So for any $n\in\N$, there exists $C_n\in \mathcal{L}_n(X_\eta)$ such that
$\#_1C_n=\max_{W\in \mathcal{L}_n(X_\eta)}\#_1W\ge nD=nd$.

Let $A_n=\#_1C_n-\#_1x^{(n)}[0,n-1]\ge 0$.

Let $|\phi|=\sup_{x\in X_\eta}|\phi(x)|$. For any $W\in\mathcal{L}_n(X_\eta)$ and $x\in W$, we have
$$\sum_{i=0}^{n-1}\phi(S^ix)\le a_{1}\#_1W+a_{01}\#_1W+a_{00}(n-2\#_1W)+2|\phi|,$$
and
$$\sum_{i=0}^{n-1}\phi(S^ix)\ge a_{1}\#_1W+a_{01}\#_1W+a_{00}(n-2\#_1W)-2|\phi|.$$

Now fix $y\in C_n$,
\begin{equation*}
\begin{aligned}
0\le& \sum_{i=0}^{n-1}\phi(S^ix^{(n)})-\sum_{i=0}^{n-1}\phi(S^iy)\\
\le& (a_1+a_{01})\#_1x^{(n)}[0,n-1]+a_{00}(n-2\#_1x^{(n)}[0,n-1])+2|\phi|\\
&-(a_1+a_{01})\#_1C_n-a_{00}(n-2\#_1C_n)+2|\phi|\\
=&-A_n(a_1+a_{01}-2a_{00})+4|\phi|\\
\le &4|\phi|.
\end{aligned}
\end{equation*}

Therefore, if $\kappa$ is Gibbs measure for $\phi$, then there exists $c>0$ such that
\begin{equation}\label{fin}
\begin{aligned}
c^{-1}\le &\kappa(C_n)\cdot2^{nP-\sum_{i=0}^{n-1}\phi(S^iy)}\\
\le &\kappa(C_n)\cdot2^{nP-\sum_{i=0}^{n-1}\phi(S^ix^{(n)})} \cdot 2^{4|\phi|}\\
\le &\nu_\eta(C_n)\cdot2^{\#_1C_n\log(1-p)}\cdot2^{nP-nd^\phi+4|\phi|}\\
\le &\nu_\eta(C_n)\cdot2^{nd\log(1-p)+nP-nd^\phi+4|\phi|},
\end{aligned}
\end{equation}
noticed that $\kappa(C_n)>0$.
We claim that $P\le -d\log(1-p)+d^\phi$. Since $a_1+a_{01}\ge 2a_{00}$, by inequality (\ref{dphi}),
$$d^\phi\ge a_{00}(1-2d)+d(a_1+a_{01}).$$
By Proposition \ref{PonX},
$$
\begin{aligned}
   P=&a_{00}(1-2d)+d\log(2^{a_1+a_{01}}+2^{2a_{00}})\\
    =&a_{00}(1-2d)+d(a_1+a_{01})-d\log(1-p)\\
    \le & d^\phi-d\log(1-p).
\end{aligned}
$$
Combined with the inequality (\ref{fin}), we have $\nu_\eta(C_n)\ge c^{-1}\cdot2^{-4|\phi|}$.
By Lemma \ref{JMlem2}, $\nu_\eta$ is atomic, which is a contradiction.
\end{proof}

\begin{proof}[Proof of Theorem \ref{t:1.2}]
	It immediately follows from Proposition \ref{p:EqState}, Theorem \ref{t:unique} and Proposition \ref{p:notGibbs}.
\end{proof}

\begin{remark}
For $\phi\in\mathscr{C}_n\setminus\mathscr{C}_2$ where $n\ge 3$, we do not know whether $\nu_\eta*B_{p,1-p}$ for some $p$ can be the equilibrium state for such $\phi$.
\end{remark}

\section{Acknowledgements}
The second author was supported by NNSF of China (11671208 and 11431012). We would
like to express our gratitude to Tianyuan Mathematical Center in Southwest China, Sichuan
University and Southwest Jiaotong University for their support and hospitality.

\end{document}